\documentclass{amsart}
\usepackage{amssymb,latexsym,amscd}

\numberwithin{equation}{section}

\newtheorem{theorem}{Theorem}[section]
\newtheorem{proposition}[theorem]{Proposition}
\newtheorem{conjecture}[theorem]{Conjecture}
\newtheorem{corollary}[theorem]{Corollary}
\newtheorem{remark}[theorem]{Remark}
\newtheorem{lemma}[theorem]{Lemma}

\newtheorem{definition}[theorem]{Definition}
\newtheorem{problem}[theorem]{Problem}

\newtheorem{maintheorem}[theorem]{Main Theorem}
\newtheorem{claim}[theorem]{Claim}

\newtheorem{acknowledgements}[theorem]{Acknowledgements}

\def\BB{\mathcal{B}}

\def\ZZ{\mathbb{Z}}

\def\CC{\mathbb{C}}

\def\RR{\mathcal{R}}

\def\II{\mathcal{I}}

\def\PP{\mathbb{P}}
\def\1{\mathbb{1}}

\def\gg{\mathfrak{g}}

\def\hh{\mathfrak{h}}

\def\OO{\mathcal{O}}

\def\PP{\mathbb{P}}

\begin{document}

\title{Braided Symmetric Algebras of simple $U_q(sl_2)$-modules  and Their Geometry}
\author{Sebastian Zwicknagl}
\maketitle

\begin{abstract}
In the present paper we prove decomposition formulae for the braided symmetric powers of simple  $U_q(sl_2)$-modules, natural quantum analogues of the classical symmetric powers of a module over a complex semisimple Lie algebra. We show that their point modules form natural non-commutative curves and  surfaces and conjecture that braided symmetric algebras give rise to an interesting non-commutative geometry, which can be viewed as a flat deformation of the geometry associated to their classical limits.   
\end{abstract}

\tableofcontents

 \section{Introduction}

Braided Symmetric and Exterior Algebras were introduced by A.~Berenstein and the author in \cite{BZ} as quantum analogs for the symmetric and exterior algebras of finite-dimensional modules over a quantized enveloping algebra $U_q(\gg)$.  One motivation for this project was to develop a new, quantum, approach to the following classical problem which is still open, and not completely settled even in the case of $\gg=sl_2(\CC)$:

\begin{problem}
\label{pr:decomp}
Let $V$ be a finite-dimensional module over a semisimple  complex Lie algebra $\gg$. Find a decomposition formula for  the symmetric and exterior powers of $V$.
\end{problem}  

The main result of the present  paper is such a decomposition formula for the  simple $U_q(sl_2)$-modules (Theorem \ref{th:main-intro} and Theorem \ref{th:main}). Let us briefly explain the definition of the braided powers and algebras, and outline their relation to the classical counterparts and some other objects in algebra and geometry. Then we will briefly outline the proof and explain some of the consequences in non-commutative geometry.
  
Recall that the category of finite-dimensional   $U_q(\gg)$-modules is braided monoidal, with  braiding $\RR_{U,V}:U\otimes V\to V\otimes U$  given by the permutation of factors composed with the action of the universal $R$-matrix. In \cite{BZ}, we define  for each finite-dimensional $U_q(\gg)$-module V, the braided exterior square $\Lambda^2_q V\subset V\otimes V$ to be the span of all "negative" eigenvectors of $\RR_{V,V}$ (i.e.~of the form $-q^{r}$, $r\in \ZZ$), and the braided symmetric square $S^2_q V\subset V\otimes V$ to be the span of all positive eigenvectors of $\RR_{V,V}$ (i.e.~of the form $q^{r}$, $r\in \ZZ$). Note that  
$$V\otimes V=\Lambda^2_q V\oplus S^2_q V\ ,$$
since $\RR_{V,V}^2$ is a diagonalizable endomorphism whose eigenvalues are powers of $q$. 
We define the braided symmetric algebra $S_q(V)$ to be the quotient of the tensor algebra $T(V)$ by the ideal generated by $\Lambda^2_q V$, and analogously for the braided exterior algebra $\Lambda_q(V)$. Note that these algebras are naturally $\ZZ_{\ge0}$ graded.  Additionally we  constructed  braided symmetric and exterior powers and were able to show that they are isomorphic as $U_q(\gg)$-modules to the respective graded components of the braided symmetric and exterior algebras.  It was shown in \cite{BZ} that the algebras $S_q(V)$ are deformations of the classical symmetric algebras  $S(\overline V)$ of the semiclassical limit  $\overline V$ of $V$ (see Section \ref{se:classical limit}). and a complete classification of the cases in which this deformation is flat was obtained by the author in \cite{ZW}.  
It is interesting to note that our construction generalizes concepts that have been present in the theory of quantum groups since its inception. Briefly after introducing quantum groups in \cite{Jimbo 1}, Jimbo constructed irreducible modules of $U_q(gl(n+1))$, using an analogue of the Schur-Weyl duality which allowed him to decompose the $m$-fold tensor products of the natural representation $\CC^{n+1}$ as modules over  Hecke-algebras $\mathcal{H}_m$ (see \cite{Jimbo 2}). The  $\mathcal{H}_m$-invariant submodule of $(\CC^{n+1})^{\otimes m}$ corresponds to our $m$-th braided symmetric power of the natural representation.

The list of  simple modules with this property  almost coincides with  the list of maximal parabolic subalgebras with Abelian nilradical, and they play an important  role in classical invariant theory (see e.g.~ Howe's article \cite{Ho}), and provide many known examples of quantized coordinate rings  and exterior powers (see e.g.~\cite{M}, \cite{Nou} and \cite{JMO}, as well as \cite{GY}). This explains why braided symmetric algebras have recently  been  used by Lehrer, Zhang and Zhang in \cite{LZZ} to study non-commutative classical invariant theory.

 Howe  studies classical invariant theory from the point of view of multiplicity free-actions, i.e.~ he considers  symmetric algebras of modules, where each graded component splits into a direct sum of pairwise non-isomorphic simple modules. In the present paper, we show that the braided symmetric algebras  of simple $U_q(sl_2)$-modules have this property, and prove an explicit decomposition formula, the following Main Theorem, where $V_k$ denotes a $(k+1)$-dimensional simple $U_q(sl_2)$-module.
 
 \begin{maintheorem}
 \label{th:main-intro}
 \noindent(a) If $\ell$ is odd, and $n\ge 3$, then
 $$ S_q ^nV_\ell\cong    \bigoplus_{i=0}^{\frac{\ell-1}{2}} V_{n\ell-4i}\ ,
 \quad \Lambda_q^n V_\ell=0   \ .$$
  \noindent(a) If $\ell$ is even, and $n\ge 3$, $p\ge4$, then
 $$S_q ^nV_\ell\cong \bigoplus_{i=0}^{\frac{n\ell}{4}} V_{n\ell-4i} \ ,\quad 
 \Lambda_q^3 V_\ell \cong \bigoplus_{i=\frac{\ell}{2}}^{\frac{3\ell-2}{4}} V_{3\ell-4i-2}\ , \quad  \Lambda_q^k V_\ell=0  \ .$$
 \end{maintheorem}
 The results for the exterior powers and the symmetric cubes were already obtained in \cite{BZ}, hence the main accomplishment of the present work is the explicit computation of the higher braided symmetric powers.
  The main idea of the proof is as follows: First, notice that the classical limit of a braided symmetric algebra $\overline{S_q(V)}$ is a Poisson algebra with the Poisson bracket defined by the standard classical $r$-matrix (see also Section \ref{se:classical limit}). It can be described as  a quotient of  $S(\overline V)$ by some ideal. The classical $r$-matrix defines a skew-symmetric quadratic bracket on   $S(\overline V)$ which does not necessarily satisfy the Jacobi identity--hence it need not be Poisson. The obstruction is an ideal generated in degree $3$, and it allows us to construct the minimal quotient of $S(\overline V)$ which is Poisson, the Poisson closure. The generators of the obstruction ideal were computed in \cite{BZ}, and in the present paper we show that $\overline{S_q(V)}$ is isomorphic to the Poisson closure of $S(\overline V)$  if $V$ is an even-dimensional simple  $U_q(sl_2)$-module, by employing results on quantum Howe-duality (see \cite{Zh}) obtained in \cite{BZ}. 

In the case of odd-dimensional simple  $U_q(sl_2)$-modules  $V_{2 k}$ we  consider the $k$-th Veronese algebra $V_q(2,k)$  of  $S_q(V_2)\cong \CC_q[x_0,x_1,x_2]$. There exists a surjective   $U_q(sl_2)$-module algebra map $S_q(V_{2 k})  \to A\subset V_q(2,k)$ where $A$ is a certain subalgebra of $V_q(2,k)$. We then easily conclude the assertion of Theorem \ref{th:main}.

The previous paragraph suggests a connection to non-commutative geometry. Following  Artin, Tate and Van Den Bergh(see e.g.~\cite{ATV}), as well as Polishchuk  \cite{Pol} one can define $Proj$ for non-commutative algebras. This notion was explored in more generality by Artin and Zhang in \cite{AZ 94} as follows. Let $A$ be a non-commutative algebra, $gr(A)$ the category of finitely  graded $A$-modules (here grading means $\ZZ$-grading) and $s$ the auto-equivalence induced by the degree-shift functor. Then, $qgr(A)$ is the quotient of the category of $gr(A)$-modules by the finite-dimensional modules (the torsion modules) and $Proj(A)$ is given by $proj(A)=(qgr(A),_A A,s)$. This allows for the definition of non-commutative point schemes.  In this paper we show that the non-commutative point-schemes attached to the braided symmetric algebras  for simple $U_q(sl_2)$-modules are  obtained from Veronese varieties. Moreover, they are deformations of the point schemes of their respective classical limits (see Section \ref{se:non-comm geom}). In the case $\ell=3$, this was first noticed by Vancliff in \cite{Van}. That the symmetric algebra of the four-dimensional simple $sl_2$-module allows no flat deformations was first published by Rossi-Doria in \cite{R-D}.    We   conjecture that braided symmetric and exterior algebras give rise to a beautiful and natural geometry in the setting of Artin and Zhang, which can be obtained from the geometry of the classical limits. 

Finally, while classical symmetric and exterior powers, in general, are given as intersections of highly non-generic subspaces--if $\Lambda^2 \overline V\subset \overline V\otimes \overline V$ was generic, then $\Lambda^3 \overline V=0$-- it appears from the results of \cite{BZ}, \cite{ZW} and our results that braided symmetric and exterior powers are as generic as the can be while still respecting the weight grading. We therefore believe that these algebras are natural objects which may shed significant insight into Problem \ref{pr:decomp}.  

The paper is organized as follows: First we recall the notions of braided symmetric algebras and their classical limits in Section \ref{se:prlime}. The following  Section \ref{se:the class theorem}  is devoted to Theorem \ref{th:main} and its proof, while Section \ref{se:Applications} considers the applications to the theory of Poisson closures and non-commutative geometry. Finally, we provide some facts about classical and quantum Veronese algebras in Appendix \ref{ap: Veronese}.

\begin{acknowledgements}
The author would like to thank A.~Berenstein for many stimulating discussions about braided symmetric algebras. However, he would have never been able to see the connection to Veronese algebras and surfaces, if M.~Vancliff had not told him about the results of her dissertation \cite{Van}, where she obtained that the point-modules of braided symmetric algebra of the irreducible $4$-dimensional representation form a Veronese curve.  She was therefore also the first person to notice that the $4$-dimensional irreducible $U_q(sl_2)$-module does not permit a symmetric algebra which is a flat deformation of the classical symmetric algebra.
Finally we would like t thank the referee for pointing us to the Artin and Zhang"s paper on non-commutative geometry \cite{AZ 94}.
 \end{acknowledgements}
  
\section{Braided Symmetric Algebras}
\label{se:prlime}
\subsection{The Quantum Group $U_q(\gg)$ and its Modules}
\label{se:q-group}

We start with the definition of the quantized enveloping algebra
associated with a complex reductive Lie algebra $\gg$ (our standard
reference here will be \cite{brown-goodearl}). Let $\hh\subset \gg$
be a Cartan subalgebra,  $P=P(\gg)$ the weight lattice and let $A=(a_{ij})$ be the Cartan matrix
for $\gg$. Additionally, let $(\cdot,\cdot)$ be the standard non-degenerate symmetric bilinear form on $\hh$.

The {\it quantized enveloping algebra} $U$ is the  $\CC(q)$-algebra generated
by the elements $E_i$ and $F_i$ for $i \in [1,r]$, and $K_\lambda $ for $\lambda \in P(\gg)$,
subject to the following relations:
$$K_\lambda  K_\mu = K_{\lambda +\mu}, \,\, K_0 = 1$$
 $$K_\lambda E_i =q^{(\alpha_i\,,\,\lambda)} E_i K_\lambda,
\,\, K_\lambda F_i =q^{-(\alpha_i\,,\,\lambda)} F_i K_\lambda
$$
for  $\lambda , \mu \in P$ and $i \in [1,r]$;  
\begin{equation}
\label{eq:upper lower relations}
E_i,F_j-F_jE_i=\delta_{ij}\frac{K_{\alpha_i}- K_{-\alpha_i}}{q^{d_i}-q^{-d_i}}
\end{equation}
for $i,j \in [1,r]$, where  $d_i=\frac{(\alpha_i\, ,\,\alpha_i)}{2}$;
and the {\it quantum Serre relations}
\begin{equation}
\label{eq:quantum Serre relations}
\sum_{p=0}^{1-a_{ij}} (-1)^p 
E_i^{(1-a_{ij}-p)} E_j E_i^{(p)} = 0,~\sum_{p=0}^{1-a_{ij}} (-1)^p 
F_i^{(1-a_{ij}-p)} F_j F_i^{(p)} = 0
\end{equation}
for $i \neq j$, where
the notation $X_i^{(p)}$ stands for the $p$-th \emph{divided power}
\begin{equation}
\label{eq:divided-power}
X_i^{(p)} = \frac{X^p}{(1)_i \cdots (p)_i}, \quad
(k)_i = \frac{q^{kd_i}-q^{-kd_i}}{q^{d_i}-q^{-d_i}} \ .
\end{equation}

%

The algebra $U$ is a $q$-deformation of the universal enveloping algebra of
the reductive Lie algebra~$\gg$, so it is commonly
denoted by $U = U_q(\gg)$.
It has a natural structure of a bialgebra with the co-multiplication $\Delta:U\to U\otimes U$
and the co-unit homomorphism  $\varepsilon:U\to \CC(q)$
given by
\begin{equation}
\label{eq:coproduct}
\Delta(E_i)=E_i\otimes 1+K_{\alpha_i}\otimes E_i, \,
\Delta(F_i)=F_i\otimes K_{-\alpha_i}+ 1\otimes F_i, \, \Delta(K_\lambda)=
K_\lambda\otimes K_\lambda \ ,
\end{equation}
\begin{equation}
\label{eq:counit}
\varepsilon(E_i)=\varepsilon(F_i)=0, \quad \varepsilon(K_\lambda)=1\ .
\end{equation}
In fact, $U$ is a Hopf algebra with the antipode anti-homomorphism $S: U \to U$ given by
\begin{equation}
\label{eq:antipode}
 S(E_i) = -K_{-\alpha_i} E_i, \,\, S(F_i) = -F_i K_{\alpha_i}, \,\,
S(K_\lambda) = K_{-\lambda}\ .
\end{equation}

Let $U^-$ (resp.~$U^0$; $U^+$) be the $\CC(q)$-subalgebra of~$U$ generated by
$F_1, \dots, F_r$ (resp. by~$K_\lambda \, (\lambda\in P)$; by $E_1, \dots, E_r$).
It is well-known that, as vectorspaces, $U=U^-\cdot U^0\cdot U^+$ (more precisely,
the multiplication map induces an isomorphism  of vectorspaces $U^-\otimes U^0\otimes U^+ \to U$).



We will consider the full sub-category  $\OO_{f}$ of the category
$U_q(\gg)-Mod$. The objects of $\OO_{f}$ are the finite-dimensional
$U_{q}(\gg)$-modules $V$ that have a weight decomposition
$$V =\oplus_{\mu\in P} V (\mu)\ ,$$
such that each $K_\lambda$ acts on each {\it weight space} $V(\mu)$ by
the multiplication with $q^{(\lambda\,|\,\mu)}$  (see e.g.,
\cite[I.6.12]{brown-goodearl}). Note that for convenience we only consider the modules of {type I} in the terminology of Jantzen \cite{Jan}, but that the results hold of course for any finite-dimensional module over $U_q(\gg)$. The category $\OO_{f}$ is semisimple
and the irreducible objects $V_\lambda$ are  generated by highest
weight spaces $V_\lambda(\lambda)=\CC(q)\cdot v_\lambda$, where
$\lambda$ is a {\it dominant weight}, i.e, $\lambda$ belongs to
$P^+=\{\lambda\in P:(\lambda\,|\,\alpha_i)\ge 0,~ \forall ~i\in
[1,r]\}$, the monoid of dominant weights.

 By definition,  the universal $R$-matrix $R \in U_{q}(\gg)\widehat \otimes U_{q}(\gg)$ (we use $\widehat \otimes$ to denote that the universal $R$-matrix lives in a certain completion of $ U_{q}(\gg) \otimes U_{q}(\gg)$(see e.g. \cite{Jan})) can be  factored as
\begin{equation}
\label{eq:Jordan}
R=R_0R_1=R_1R_0
\end{equation}
 where $R_0$ is ''the diagonal part'' of $R$, and  $R_1$ is unipotent, i.e., $R_1$ is a formal power series
\begin{equation}
\label{eq:R1}
R_1=1\otimes 1+(q-1)x_1+ (q-1)^2x_2+\cdots \ ,
\end{equation}
where all $x_k\in {U'}^-\otimes_{\CC[q,q^{-1}]} {U'}^+$, where
${U'}^-$ (resp.~${U'}^+$) is the integral form of $U^-$ (resp.~$U^+$); i.e.,
${U'}^-$ is the $\CC[q,q^{-1}]$-subalgebra of $U_q(\gg)$ generated by
all $F_i$ (resp. by all $E_i$). 
It is also well known that if $U,V$ are objects of  $\OO_f$, and  $u_\lambda\in U(\lambda)$ and  $v_\mu\in V(\mu)$ are highest weight elements, then 
$R_0(u_\lambda\otimes v_\mu)=q^{(\lambda\,|\,\mu)}u_\lambda\otimes v_\mu$.

Let  $R^{op}=\tau (R)$, where  $\tau:U_{q}(\gg)\widehat \otimes U_{q}(\gg)\to U_{q}(\gg)\widehat \otimes U_{q}(\gg)$
is the permutation of factors. Clearly, $R^{op}=R_0R_1^{op}=R_1^{op}R_0$.

Following \cite[Section 3]{DR1},  define $D\in  U_{q}(\gg)\widehat \otimes U_{q}(\gg)$ by
\begin{equation}
\label{eq:D}
D:=R_0\sqrt{R_1^{op}R_1}=\sqrt{R_1^{op}R_1}R_0 \ .
\end{equation}

Clearly, $D$ is well-defined because $R_1^{op}R_1$ is also unipotent as well as its square root.  By definition,
$D^{2}=R^{op}R$, $D^{op}R=RD$.

Furthermore, define
\begin{equation}
\label{eq:hat R}
\widehat R:=RD^{-1}=(D^{op})^{-1}R=R_1\left(\sqrt{R_1^{op}R_1}\right)^{-1}
\end{equation}
It is easy to see that
\begin{equation}
\label{eq:unitary}
 \widehat R^{op}=\widehat R^{\,\,-1}
\end{equation}
According to \cite[Proposition 3.3]{DR1}, the pair $(U_{q}(\gg), \widehat R)$ is a {\it coboundary} Hopf algebra. For definition and properties of coboundary Hopf algebras see e.g.~\cite{ZW1}.

The category $\OO_{f}$ is naturally braided, with braiding  defined by $\RR_{U,V}:U\otimes V\to V\otimes U$, where
$$\RR_{U,V}(u\otimes v)=\tau R(u\otimes v)$$
for any $u\in U$, $v\in V$, with  $\tau:U\otimes V\to V\otimes U$, as above, the ordinary permutation of factors.

Denote by  $C\in Z(\widehat{U_{q}(\gg)})$ (once again, we have to work in a completion of  $U_{q}(\gg)$) the {\it quantum
Casimir} element which acts on any irreducible $U_q(\gg)$-module
$V_\lambda$ in $\OO_f$ by the scalar multiple
$q^{(\lambda\,|\,\lambda+2\rho)}$, where $2\rho$ is the sum of
positive roots.

The following fact is  well-known.
\begin{lemma}
\label{le:braiding casimir} We have $\RR^2=\Delta(C^{-1})\circ
(C\otimes C)$. In particular, for each $\lambda,\mu,\nu\in  P_+$ the
restriction of $\RR^2$ to the $\nu$-th isotypic component
$I^\nu_{\lambda,\mu}$ of the tensor product $V_\lambda\otimes V_\mu$
is scalar multiplication by
$q^r$ where   $r=(\lambda\,|\,\lambda)+(\mu\,|\,\mu)-(\nu\,|\,\nu)+(2\rho\,|\,\lambda+\mu-\nu)$.

\end{lemma}

We now define the diagonalizable $\CC(q)$-linear  map
$D_{U,V}:U\otimes V\to U\otimes V$ by $D_{U,V}(u\otimes v)=
D(u\otimes v)$ for any objects $U$ and $V$ of $\OO_f$. It is easy to
see that the operator $D_{V_\lambda,V_\mu}:V_\lambda\otimes V_\mu\to
V_\lambda\otimes V_\mu$ acts on the $\nu$-th isotypic component
$I^\nu_{\lambda,\mu}$ in $V_\lambda\otimes V_\mu$ by the scalar
multiplication with
$q^{\frac{r}{2}}$ with $r=(\lambda\,|\,\lambda)+(\mu\,|\,\mu)-(\nu\,|\,\nu)+(2\rho\,|\,\lambda+\mu-\nu)$ as in Lemma \ref{le:braiding casimir}.

The coboundary Hopf algebra  $(U_{q}(\gg), \widehat R)$ provides $\OO_f$ with the structure of a coboundary or cactus category in the following way.
For any  $U$ and $V$ in $\OO_f$ define the {\it normalized braiding} $\sigma_{U,V}$ by
\begin{equation}
\label{eq:sigma}
\sigma_{U,V}(u\otimes v)=\tau \widehat R (u\otimes v) \ ,
\end{equation}

Therefore, we have by (\ref{eq:hat R}):
\begin{equation}
\label{eq:formula for sigma}
\sigma_{U,V}=D_{V,U}^{-1} \RR_{U,V}=\RR_{U,V}D_{U,V}^{-1} \ .
\end{equation}

We will sometimes write $\sigma_{U,V}$ in a more explicit way:
\begin{equation}
\label{eq:sigma root}
\sigma_{U,V}=\sqrt{\RR_{V,U}^{-1}\RR_{U,V}^{-1}} \RR_{U,V}=\RR_{U,V}\sqrt{\RR_{U,V}^{-1}\RR_{V,U}^{-1}}
\end{equation}


The following fact is an obvious corollary of (\ref{eq:unitary}).
\begin{lemma}
\label{le:symm comm constraint} $\sigma_{V,U}\circ
\sigma_{U,V}=id_{U\otimes V}$ for any  $U,V$ in $\OO_f$. That is,
$\sigma$ is a symmetric commutativity constraint.
\end{lemma}


%
%


\begin{remark} If one replaces the braiding  $\RR$ of $\OO_f$  by its inverse $\RR^{-1}$,
the symmetric commutativity constraint $\sigma$ will not change.
\end{remark}

\subsection{Braided Symmetric and Exterior Powers}
\label{se:BESP}
 In this section we will recall the definitions and some basic properties of the braided symmetric and exterior powers and algebras introduced in \cite{BZ}. 

For any morphism $f:V\otimes V\to V\otimes V$ in $\OO_f$ and $n>1$
we denote by $f^{i,i+1}$, $i=1,2,\ldots,n-1$ the morphism
$V^{\otimes n}\to V^{\otimes n}$ which acts as $f$ on the $i$-th and
the $(i+1)$st factors. Note that $\sigma_{V,V}^{i,i+1}$ is always an
involution on $V^{\otimes n}$.

\begin{definition}
\label{def:symmetric power} For an object $V$ in $\OO_f$ and $n\ge
0$ define the {\it braided symmetric power} $S_\sigma^nV \subset
V^{\otimes n}$ and the {\it braided exterior power}
$\Lambda_\sigma^nV\subset V^{\otimes n}$  by:
$$S_\sigma^nV=\bigcap_{1\le i\le n-1}  Ker~( \sigma_{i,i+1}-id)=\bigcap_{1\le i\le n-1} Im~(\sigma_{i,i+1}+id)\ ,$$
$$\Lambda_\sigma^nV=\bigcap_{1\le i\le n-1} Ker~ (\sigma_{i,i+1}+id)=\bigcap_{1\le i\le n-1} Im~(\sigma_{i,i+1}-id) ,$$
where we abbreviate $\sigma_{i,i+1}=\sigma_{V,V}^{i,i+1}$.
\end{definition}

\begin{remark} Clearly, $-\RR$ is also a braiding on $\OO_f$ and $-\sigma$ is the corresponding normalized braiding.
Therefore, $\Lambda_\sigma^nV=S_{-\sigma}^nV$ and
$S_\sigma^nV=\Lambda_{-\sigma}^nV$. That is, informally speaking,
the symmetric and exterior powers are mutually ''interchangeable''.

\end{remark}

\begin{remark}
\label{re: BSA via braiding}
 Another way to introduce the symmetric and exterior squares involves the well-known fact that the braiding $\RR_{V,V}$
is a semisimple operator $V\otimes V\to V\otimes V$, and all the
eigenvalues of $\RR_{V,V}$ are of the form $\pm q^r$, where $r\in
\ZZ$. Assume that $q\in \mathbb{R}_{> 0}$. Then the  positive eigenvectors of $\RR_{V,V}$ span
$S_\sigma^2V$ and the negative eigenvectors  of $\RR_{V,V}$ span
$\Lambda_\sigma^2V$.
\end{remark}

Clearly, $S_{\sigma}^{0}V=\CC(q)\ ,\ S_{\sigma}^{1}V=V\ , \Lambda_{\sigma}^{0}V=\CC(q)\ ,\ \Lambda_{\sigma}^{1}V=V$, and
$$S_\sigma^2V=\{v\in V\otimes V\,|\, \sigma_{V,V}(v)=v\},~\Lambda_\sigma^2V=\{v\in V\otimes V\,|\, \sigma_{V,V}(v)=-v\} \ .$$





The following fact is obvious.

\begin{proposition}
\label{pr:injections}  The association $V\mapsto
S^n_\sigma V$ is a functor from $\OO_f$ to $\OO_f$, for each $n\ge 0$,  and, similarly,  the
association $V\mapsto \Lambda_\sigma^n V$ is a functor from $\OO_f$
to $\OO_f$. In particular, an embedding $U\hookrightarrow V$ in the
category $\OO_f$ induces injective morphisms
$$S_{\sigma}^n U \hookrightarrow S_{\sigma}^n V, ~ \Lambda_{\sigma}^n U\hookrightarrow \Lambda_{\sigma}^n V\ .$$

\end{proposition}

\begin{definition}
\label{def:symmetric algebra}
For any  $V\in Ob(\OO)$   define the {\it braided symmetric algebra} $S_\sigma(V)$ and the {\it braided  exterior algebra} $\Lambda_\sigma(V)$ by:
\begin{equation}
\label{eq:injections}
S_\sigma(V)=T(V)/\left< \Lambda_\sigma^{2}V\right>,~\Lambda_\sigma(V)=T(V)/ \left<S_\sigma^{2}V\right> \ ,
\end{equation}
where $T(V)$ is the tensor algebra of $V$ and $\left<I\right>$ stands for the two-sided  ideal in $T(V)$ generated by a subset $I\subset T(V)$.
\end{definition}

Note that the algebras $S_\sigma(V)$ and $\Lambda_\sigma(V)$ carry a natural $\ZZ_{\ge 0}$-grading:
$$S_\sigma(V)=\bigoplus_{n\ge 0} S_\sigma(V)_n, ~~~\Lambda_\sigma(V)=\bigoplus_{n\ge 0} \Lambda_\sigma(V)_n\ ,$$
since the respective ideals in $T(V)$ are homogeneous.

\smallskip

Denote by  $\OO_{gr,f}$ the sub-category of $U_q(\gg)-Mod$ whose objects are $\ZZ_{\ge 0}$-graded:
$$V=\bigoplus_{n\in \ZZ_{\ge 0}}  V_n\ ,$$

where each $V_n$ is an object of $\OO_f$; and morphisms are those homomorphisms of  $U_q(\gg)$-modules which preserve the $\ZZ_{\ge 0}$-grading.

Clearly,  $\OO_{gr,f}$ is a tensor category under the natural
extension of the tensor structure of $\OO_f$. Therefore, we can
speak of algebras and co-algebras in $\OO_{gr,f}$.

%
%
%

By the very definition,  $S_{\sigma}(V)$ and $\Lambda_{\sigma}(V)$ are algebras in $\OO_{gr,f}$.

\begin{proposition} The assignments $V\mapsto S_\sigma(V)$ and $V\mapsto \Lambda_\sigma(V)$ define functors from $\OO_f$
to the category of algebras in $\OO_{gr,f}$.
\end{proposition}

We conclude the section with two important features of braided symmetric exterior powers and algebras.

\begin{proposition}\cite[Prop.2.13]{BZ}
\label{pr:n-th symm-power} For any  $V$ in $\OO_{f}$ each embedding
$V_{\lambda}\hookrightarrow V$ defines embeddings
$V_{n\lambda}\hookrightarrow S_{\sigma}^n V $ for all  $n\ge 2$. In
particular, the  algebra $S_{\sigma}(V)$ is infinite-dimensional.
\end{proposition}

\begin{proposition}\cite[Prop.2.11 and Eq. 2.3]{BZ}
Let $V$ be an object of $\OO_f$ and $V^*$ its dual in $\OO_f$. We have the following $U_q(\gg)$-module isomorphisms.
\begin{equation}
\label{eq:double duals}
(S_\sigma^n V^{*})^*\cong S_\sigma(V)_n, ~(\Lambda_\sigma^n V^{*})^*\cong \Lambda_\sigma(V)_n  \ .
\end{equation}

\end{proposition}



\subsection{The Classical Limit of Braided Algebras}
\label{se:classical limit}

In this section we will discuss the specialization of the braided symmetric and exterior algebras at $q=1$, the classical limit. All of the results in this section are either well known or proved in \cite{BZ}. For a more detailed discussion of the classical limit we refer the reader to \cite[Section 3.2]{BZ}. The explicit construction of the classical limit involves the consideration of integral lattices in $U_q(\gg)$ and its modules, which we will omit here for brevity's sake (see e.g.~\cite[Section 3.2]{BZ}, or for a shorter version \cite[Section 4.3]{ZW}).

\begin{definition}
 Let $\mathcal{C}$ and $\mathcal D$ be two categories. We say that a functor $F:\mathcal C\to \mathcal D$ is  an {\it almost equivalence} of $\mathcal C$ and $\mathcal D$ if:

\noindent (a)  $F(c)\cong F(c')$ in $\mathcal D$ implies that $c\cong c'$ in  ${\mathcal C}$  for any two objects $c,c'$ of $\mathcal C$ ;

\noindent (b) for any object  $d$ in $\mathcal{D}$ there exists an object $c$ in ${\mathcal C}$ such that $F(c)\cong d$ in ${\mathcal D}$.

\end{definition} 
 
Denote by $\overline \OO_f$ the full (tensor) sub-category category of $U(\gg)-Mod$, whose objects $\overline V$ are finite-dimensional $U(\gg)$-modules having a weight decomposition $\overline V=\oplus_{\mu\in P} \overline V(\mu)$.  We have the  following fact.

 \begin{proposition}\cite[Cor 3.22]{BZ}
\label{pr:almost equivalent O bar O}
The categories  $\OO_f$ and  $\overline \OO_f$ are almost equivalent. Under this almost equivalence a simple $U_q(\gg)$-module $V_\lambda$ is mapped to the simple $U(\gg)$-module $\overline V_\lambda$.
\end{proposition}

Let $V\cong \bigoplus_{i=1}^{n} V_{\lambda_{i}}  \in \OO_{f}$.  We call $\overline V \cong \bigoplus_{i=1}^{n} \overline V_{\lambda_{i}}  \in \overline \OO_{f}$ the {\it classical limit} of $V$ under the  almost equivalence of Proposition \ref{pr:almost equivalent O bar O}.

The  following result relates the classical limit of braided symmetric algebras and Poisson algebras.

\begin{theorem}\cite[Theorem 2.29]{BZ}
\label{th: Poisson-closure}
\label{th:flat poisson closure} Let $V$ be an object of $\OO_f$ and let a  $\overline V$ in $\overline \OO_f$ be   the classical limit of $V$. Then:

\noindent(a)The classical limit $\overline{S_\sigma(V)}$ of the braided symmetric algebra $S_\sigma(V)$ is a quotient of the symmetric algebra $S(\overline V)$. In particular, $\dim_{\CC(q)} S_\sigma(V)_n=\dim_{\CC}(\overline{S_\sigma(V)})_n\le \dim_{\CC} S(\overline V)_n$.

\noindent(b)Moreover, $\overline{S_\sigma(V)}$ admits a Poisson structure defined by  $\{u,v\}=r^-(u\wedge v)$, where $r^-$ is an anti-symmetrized  classical $r$-matrix.
\end{theorem}

For more on classical $r$-matrices and their relation to the $R$-matrices associated to quantum groups, we refer the reader to \cite{BZ}, \cite{ZW1} or \cite{ES}. 








\section{Braided Symmetric Powers of $U_q(sl_2)$-Modules}
\label{se:the class theorem}

In this section, we will prove our main result. 
  Recall that the irreducible $U_q(sl_2)$-modules are labeled by non-negative integers, and that $\dim(V_\ell)=\ell+1$. The module $V_\ell$ has a natural weight basis $v_0,v_1,\ldots, v_\ell$ such that $K(v_i)=q^{\ell-i} v_i$ and $E(v_i)=[i]_q v_{i-1}$ where $[i]_q$ denotes the quantum number  $[i]_q= \frac{q^{i}-q^{-i}}{q-q^{-1}}$.
Moreover note that it follows from  classical Lie theory and the definition of braided powers (Definition \ref{def:symmetric power}) that 
$$ V_\ell\otimes V_\ell=\Lambda^2_\sigma V_\ell\oplus S^2_\sigma V_\ell\ ,
\quad S^2_\sigma V_\ell=\bigoplus_{i=0}^{\frac{\ell}{2}}V_{2\ell-4i}\ ,\ \text{and}\quad \Lambda^2_\sigma V_\ell= \bigoplus_{i=0}^{\frac{\ell-1}{2}}V_{2\ell-2-4i}\ .$$

We will now state our main result.

\begin{theorem}
\label{th:main}
\noindent(a) Let $\ell$ be odd and let $V_\ell$ be the $(\ell+1)$-dimensional irreducible $U_q(sl_2)$-module. The $n$-th braided symmetric power splits as
$$ S^n_\sigma (V_\ell)\cong \bigoplus_{i=0}^{\frac{\ell-1}{2}} V_{n\ell-4i}\ .$$

\noindent(b)  Let $\ell$ be even and let $V_\ell$ be the $\ell+1$-dimensional irreducible $U_q(sl_2)$-module. The $n$-th braided symmetric power splits as
$$ S^n_\sigma (V_\ell)\cong \bigoplus_{i=0}^{\frac{n\ell}{4}} V_{n\ell-4i}\ .$$
\end{theorem}

\begin{remark}
The above decomposition was originally conjectured in \cite[Conjecture 2.37]{BZ}, based on computer calculations. \end{remark}

\begin{proof}
We first recall the main result of the paper \cite{BZ} which established the Theorem in the case $n=3$.

\begin{proposition} \cite[Theorem 2.35]{BZ}
\label{pr:symmetric cubes}
\noindent(a) Let $\ell$ be odd and let $V_\ell$ be the $\ell+1$-dimensional irreducible $U_q(sl_2)$-module. The third braided symmetric power splits as
$$ S^3_\sigma (V_\ell)\cong \bigoplus_{i=0}^{\frac{\ell-1}{2}} V_{3\ell-4i}\ .$$

\noindent(b)  Let $\ell$ be even and let $V_\ell$ be the $\ell+1$-dimensional irreducible $U_q(sl_2)$-module. The third braided symmetric power splits as
$$ S^3_\sigma (V_\ell)\cong \bigoplus_{i=0}^{\frac{3\ell}{4}} V_{3\ell-4i}\ .$$
\end{proposition}

  \subsection{Proof of  Theorem \ref{th:main} (a)}
 Note  that the assertion holds in the case $n=2$, and in the case $n=3$ by Proposition \ref{pr:symmetric cubes}(a).

We first prove that one can find copies of the modules $V_{n\ell-4i}$, $i=0,\ldots, \frac{\ell-1}{2}$ in $S_\sigma (V_\ell)_n$. Observe that if $v\in V_\ell\otimes V_\ell$ is a highest weight vector of weight $m\in\ZZ$  , then $v\otimes v_0^{\otimes n-2}\in V_\ell^{\otimes n}$ is a highest weight vector  and its weight  is $(m+(n-2)\ell)$. We obtain the following result.

\begin{proposition}
\label{pr: hishgest weights in Sqs}
  Let ${\bf v}_i$ be a highest weight vector in $S^2_\sigma V_\ell\subset V_\ell\otimes V_\ell$ of weight $2\ell-4i$ for $0\le i\le \frac{\ell-1}{2}$. Let  ${\bf v}_i^n={\bf v}_i\otimes v_0^{\otimes n-2}\in V_\ell^{\otimes n}$.

\noindent(a) ${\bf v}_i^n$ is a highest weight vector of weight $n\ell-4i$.

\noindent(b) Moreover, ${\bf v}_i^n  \notin  \langle \Lambda^2_\sigma V_\ell\rangle_n\subset V_\ell^{\otimes n}$  
where, as above, $\langle \Lambda^2_\sigma V_\ell\rangle$ denotes the ideal generated by $\Lambda^2_\sigma V_\ell$ in $T(V_\ell)$.

\end{proposition}
\begin{proof}
Prove (a) first. Observe that if $v\in V_\ell\otimes V_\ell$ is a highest weight vector of weight $m\in\ZZ$  , then $v\otimes v_0^{\otimes n-2}\in V_\ell^{\otimes n}$ is a highest weight vector  and its weight  is $(m+(n-2)\ell)$.  Part(a) follows.

Prove (b) next.
 The vector ${\bf v}_i^n$ is clearly invariant under all $\sigma_{j,j+1}$ for $j\ge 3$, hence it suffices to prove the assertion  in the case $n=3$. To do so, we have to employ the results and methods from \cite[Section 3.4]{BZ}, where all notions and constructions are explained in detail. However, it  seems to long to introduce here (approximately ten pages), so we will recall only the most important results. Using Howe duality, we identify the set $\left(V_\ell^{\otimes 3}(3\ell-4i)\right)^+$ of highest weight vectors of weight $3\ell-4i$ with the weight space $V_{3\ell-2i,2i,0}({\bf0})$ of the $gl_3$-module of highest weight $(3\ell-2i,2i,0)$. Note that $V_{3\ell-2i,2i,0}({\bf0})$ has odd dimension. We additionally construct two bases $\BB_1=\{b_1^1,\ldots, b_{2m+1}^1\}$ and $\BB_2=\{b_1^2,\ldots, b_{2m+1}^2\}$ in $\left(V_\ell^{\otimes 3}(3\ell-4i)\right)^+$  such that $S^2_\sigma V_\ell\otimes V\cap V_\ell^{\otimes 3,+}(3\ell-4i)$ is spanned by $\{b_1^1, b_3^1,\ldots, b_{2m+1}^1\}$ and   $\Lambda^2_\sigma V_\ell\otimes V\cap V_\ell^{\otimes 3,+}(3\ell-4i)$ is spanned by $\{b_2^1, b_4^1,\ldots, b_{2m}^1\}$, whereas $V\otimes  S^2_\sigma V_\ell\cap \left(V_\ell^{\otimes 3}(3\ell-4i)\right)^+$ and $V\otimes  \Lambda^2_\sigma V_\ell\cap \left(V_\ell^{\otimes 3}(3\ell-4i)\right)^+$ are spanned by $\{b_1^2, b_3^2,\ldots, b_{2m+1}^2\}$, resp.~$\{b_2^2, b_4^2,\ldots, b_{2m}^2\}$,  for more details see \cite[Theorem 3.40 and Lemma 3.55]{BZ}. Moreover, we can identify ${\bf v}_i^3$ with the basis vector $b_1^1$, since it corresponds to the $\alpha_1$-string of minimal length passing through $V_{(3\ell-2i,2i,0)}({\bf0})$ (see \cite[(3.19)]{BZ}).  Note that \cite[Theorem 3.43]{BZ} implies that the set 
 $\{b_2^1, b_4^1,\ldots, b_{2m}^1, b_2^2, b_4^2,\ldots, b_{2m}^2\}$ spans 
 \begin{equation}
 \label{eq:}
 \left(\Lambda^2_\sigma V_\ell\otimes V+V\otimes  \Lambda^2_\sigma V_\ell\right)\cap \left(V_\ell^{\otimes 3}(3\ell-4i)\right)^+=\langle  \Lambda^2_\sigma V_\ell\rangle_3\cap \left(V_\ell^{\otimes 3}(3\ell-4i)\right)^+\ne \left(V_\ell^{\otimes 3}(3\ell-4i)\right)^+\  .
 \end{equation} 
  We then obtain from \cite[Theorem 3.43]{BZ} that   the set   $ \{b_1^1\}\cup\{b_2^1, b_4^1,\ldots, b_{2m}^1, b_2^2, b_4^2,\ldots, b_{2m}^2\}$ is linearly independent. Equation  and hence that  ${\bf v}_i^3\notin \langle  \Lambda^2_\sigma V_\ell\rangle_3$. Proposition \ref{pr: hishgest weights in Sqs} is proved. 
 \end{proof}
 
 Denote by $W_n=\bigoplus_{i=0}^{\frac{\ell-1}{2}} V_{n\ell-4i}$ and by  $\overline W_n=\bigoplus_{i=0}^{\frac{\ell-1}{2}} \overline  V_{n\ell-4i}$.   
 Proposition \ref{pr: hishgest weights in Sqs}  yields  natural embeddings 
 \begin{equation}
 \label{eq:embedding}
 W_n \hookrightarrow S_\sigma (V_\ell)_n\ ,\quad  \overline W_n=\bigoplus_{i=0}^{\frac{\ell-1}{2}} \overline  V_{n\ell-4i}\hookrightarrow  \overline {S_\sigma (V_\ell)}_n\ .\end{equation}

Suppose from now on that $\ell$ is odd.  We have the following fact.
  \begin{lemma}
  \label{le:dimensions lower}
 (a)   $dim_{\CC(q)}(W_n)=dim_{\CC}(\overline W_n)=\binom{\ell+2}{2}+(n-2)\binom{\ell+1}{2}$.\\
  
  (b) $dim_{\CC(q)}(S_\sigma (V_\ell)_n)=dim_\CC(\overline {S_\sigma (V_\ell)}_n)\ge \binom{\ell+2}{2}+(n-2)\binom{\ell+1}{2}$. 
  \end{lemma}
  \begin{proof}
  We compute
  $$dim_{\CC(q)}(W_n)=\sum_{i=0}^{\frac{\ell-1}{2}} (n\ell+1-4i)= 
  \sum_{i=0}^{\frac{\ell-1}{2}} (n-2)\ell+\sum_{i=0}^{\frac{\ell-1}{2}} (2\ell+1-4i)     $$
  $$   = (n-2)\ell\frac{\ell+1}{2} +(\ell+2) \frac{\ell+1}{2}+(\ell-1)\frac{\ell+1}{2}+4\binom{\frac{\ell+1}{2}}{2}=(n-2)\binom{\ell+1}{2}+\binom{\ell+2}{2}\ .$$
   We compute $dim_{\CC}(\overline W_n)$, analogously. Part (b) follows from the embeddings in Equation \ref{eq:embedding}.  
  \end{proof}

Now, let $\gg$ be a finite-dimensional complex  simple Lie algebra and let ${\bf c}\in \Lambda^3(\gg)$ be the canonical element associated to the Lie bracket (see e.g.~\cite{ZW}). If $ \overline V$ is a $U(\gg)$-module, then $\gg^{\otimes 3} \subset U(\gg)^{\otimes 3} $ acts component-wise on $ \overline V^{\otimes 3}$. This action defines a map ${\bf c}: \overline  V^{\otimes 3}\to  \overline V^{\otimes 3}$. It is easy to see that $ {\bf c}(\Lambda^3 \overline  V)\subset (S  (\overline V))_3$ . We call this map the {\it Jacobian} map.  It is well known that ${\bf c}$ is $\gg$-invariant and we showed in \cite{ZW} that the Jacobian map is therefore a homomorphism of $U(\gg)$-modules.
We prove in \cite{ZW} and \cite{BZ} that its image generates the {\it Jacobian ideal} $J$, the obstruction which prevents the bracket defined by $r^-$ from satisfying the Jacobi identity (see also Section \ref{se:Poisson closure}). The ideal $J$ is clearly a $\ZZ$-graded $U(\gg)$-module ideal.  We refer to the quotient $S( \overline V)/J$ as the {\it Poisson closure} and we showed in \cite{BZ} that the natural projection map $S(V)\to \overline{S_\sigma(V)}$ factors through  $S( \overline V)/J$. 
Hence, $dim_\CC(S( \overline V)/J)_n)\ge dim_\CC(\overline{S_\sigma(V)}_n)$.

Now, let $\gg=sl_2$. Then $r^-=E\wedge F$ and ${\bf c}=E\wedge F\wedge H$ (see e.g.~\cite{BZ}). Let $\overline V_\ell$ be the standard $(\ell+1)$-dimensional simple $U(sl_2)$-module and let $\overline v_0,\overline v_1,\ldots,\overline v_\ell$ be the standard weight basis of  $\overline V_\ell$; i.e., $H(\overline v_i)=(\ell-2i) \overline v_i$, $E(\overline v_i)=i \overline v_{i-1}$ and $F(\overline v_i)=(\ell-i)\overline v_{i+1}$ (here we assume that $\overline v_{-1}=\overline v_{\ell+1}=0$).
We can now define for $k\ge 0$ the standard weight basis of $\Lambda^k \overline V_\ell$ resp.~$(S( \overline V_\ell))_k$  as
$\{\overline v_{i_1}\wedge \overline v_{i_2}\wedge \ldots\wedge \overline v_{i_k}|0\le i_1<i_2<\ldots<i_k\le \ell\}$, resp.~$\{\overline v_{i_1}\cdot \overline v_{i_2}\cdot \ldots\cdot \overline v_{i_k}|0\le i_1\le i_2\le\ldots\le i_k\le \ell\}$. Consider the  lexicographic ordering on the basis of $(S( \overline V_\ell))_k$. Denote by $Te(\overline x)$ the least or terminal monomial of an element $\overline x\in (S(V_\ell))_k$. We have the following fact.

\begin{lemma}
Let $\ell$ be odd. Then the Jacobian map ${\bf c}: \Lambda^3 \overline  V\to  (S  (\overline V))_3 $ is injective and if ${\bf \overline v}=\overline v_a\wedge \overline v_b\wedge  \overline v_c$ with $0\le a<b<c\le \ell$ then   $Te({\bf c}({\bf \overline v}))=\overline v_{a+1}\overline v_{b}\overline v_{c-1}$.
\end{lemma}

\begin{proof}

The second assertion follows from the following direct computation:
$$ {\bf c}({\bf \overline v})=(E\wedge F\wedge H)({\bf \overline v})=$$
 $$=(\ell-a)(\ell-2b)c\  \overline v_{a+1}\overline v_{b}\overline v_{c-1}+(\ell-2a)(\ell-b)c\  \overline v_{a}\overline v_{b+1}\overline v_{c-1} +(\ell-a)b(\ell-2c)  \overline v_{a+1}\overline v_{b-1}\overline v_{c}$$
$$ +a(\ell-b)(\ell-2c) \overline v_{a-1}\overline v_{b+1}\overline v_{c}+(\ell-2a)b(\ell-c) \overline v_{a}\overline v_{b+1}\overline v_{c-1}+a(\ell-2b)(\ell-c)\overline v_{a-1}\overline v_{b}\overline v_{c+1}\ .$$
Clearly, $Te({\bf c}({\bf \overline v}))=\overline v_{a+1}\overline v_{b}\overline v_{c-1}$,  as the coefficient $(\ell-a)(\ell-2b)c\ne 0$ for odd $\ell$. Indeed,  one has $(\ell-2b)\ne 0$,  $a\le \ell-2$ and $c\ge 2$.    

The first assertion now follows directly from the fact that the map $\ZZ^3\to \ZZ^3$ which sends $(a,b,c)$ to $(a-1,b,c+1)$ is a bijection. Hence, the set 
$$\{ {\bf c}({\bf \overline v})|{\bf \overline v}=\overline v_a\wedge \overline v_b\wedge \overline v_c, 0\le a<b<c\le \ell\}\subset (S(\overline V_\ell))_3$$ consisting of  images of the standard basis vectors of $\Lambda^3\overline V_\ell$,  is linearly independent, as each element has a different terminal monomial. The Jacobian  map is therefore injective. The lemma is proved.
\end{proof}
 
Observe that every standard basis element $\overline v_a \cdot \overline v_b\cdot \overline v_c\in (S(\overline V_\ell))_3$ with $1\le a\le b\le c\le \ell-1$ is the terminal monomial of some element  $\overline w$ in the image of the Jacobian map. Let ${\bf \overline v}'=\prod_{j=0}^\ell \overline v_j^{k_j}\in (S(\overline V_\ell))_{n-3}$; i.e., $\sum_{j=0}^\ell k_j=n-3$. Then $\overline w\cdot {\bf \overline v}'\in J_n$.
We obtain that  a standard basis element  ${\bf \overline v}=\prod_{j=0}^\ell\overline  v_j^{k_j}\in (S(\overline V_\ell))_n$ (with $\sum_{j=0}^\ell k_j=n$) is the terminal monomial of some element in the   Jacobian ideal  if $k_0+k_\ell\le n-3$. 
   
Denote by $T_{n,\ell}$ the complementary set of standard basis vectors, namely
$$ T_{n,\ell}=\{ {\bf \overline v}=\prod_{j=0}^\ell \overline v_j^{k_j}\in (S(V_\ell))_n|\sum_{j=0}^\ell k_j=n\ \text{and}\ (k_0+k_\ell)\ge n-2\} \ .$$
Observe that $dim_\CC((S(\overline V_\ell)/J)_n)\le |T_{n,\ell}|$, where $|\mathfrak{S}|$ denotes the cardinality of a set $\mathfrak{S}$.
Theorem \ref{th:main} (a) now follows from Lemma \ref{le:dimensions lower} (b) and the following claim. 
\begin{claim}
(a) The cardinality of $T_{2,\ell}$ is $|T_{2,\ell}|=\binom{\ell+2}{2}$.\\
(b) The difference of cardinalities of $T_{n,\ell}$ and $T_{n-1,\ell}$ is $|T_{n,\ell}| -|T_{n-1,\ell}|=\binom{\ell+1}{2}$ for $n\ge 2$. \\
(c) The cardinality of $T_{n,\ell}$ is $|T_{n,\ell}|=\binom{\ell+2}{2}+(n-2)\binom{\ell+1}{2}$.
\end{claim}

\begin{proof}
Part (a) is obvious, as  $T_{2,\ell}$ is a basis of $(S(\overline V_\ell))_2$. Now, let us prove part (b). The case $n=2$ is obvious, as well. Suppose $n\ge 3$. Denote by $U_{2,\ell}$ the set 
$U_{2,\ell}=\{{\bf \overline v}'=\overline v_a\cdot \overline v_b\in T_{2,\ell}| a,b\ne 0\}$, by  $U_{n,\ell}$ the set $U_{n,\ell}=\{{\bf \overline v}={\bf \overline v}'\cdot \overline v_\ell^{n-2}|{\bf\overline  v}'\in U_{2,\ell}\}$, and by $T_{n,\ell}^0$ the set $T_{n,\ell}^0=\{{\bf \overline v}={\bf \overline v}'\cdot\overline v_0|{\bf \overline v}'\in T_{n-1,\ell}\}$. 
Observe that $T_{n,\ell}=T_{n,\ell}^0\sqcup U_{n,\ell}$. It is easy to see that  the cardinalities satisfy the relation
$|U_{n,\ell}|=|U_{2,\ell}|$ and $|U_{2,\ell}|=|T_{2,\ell}|-(\ell+1)=\binom{\ell+1}{2}$. Part (b) follows, and
part (c) is an immediate corollary from (a) and (b).
\end{proof}
 
  Theorem \ref{th:main}(a)  is proved.

  \subsection{Proof of  Theorem \ref{th:main} (b)}

It remains to prove part (b). First, recall  the definition and properties of the classical and quantum Veronese algebras from Appendix \ref{ap: Veronese}, in particular the natural $U(sl_2)$, resp.~$U_q(sl_2)$-module algebra structure. Notice that  if $\ell$ is even, then we can embed $V_{\ell}$ into $S_\sigma (V_2)_{\frac{\ell}{2}}$ and we obtain a homomorphism from  $S_\sigma V_\ell$ to $V_q(2,\frac{\ell}{2})$ (see Lemma \ref{le:V gen subal of Ve}).  Its image is the subalgebra $A(2,\frac{\ell}{2})$, introduced in  Appendix \ref{ap: Veronese}. Proposition \ref{pr:symmetric cubes}, Lemma \ref{le:V generates sub Vero}(a)  and Lemma  \ref{le:Veronese splitting} now imply that $A(2,\frac{\ell}{2})_2\cong (S_\sigma(V_\ell))_2$ and   $A(2,\frac{\ell}{2})_3\cong (S_\sigma(V_\ell))_3$.  The kernel is an ideal generated by homogeneous elements in degrees $2$ and $3$, by Lemma \ref{le:V generates sub Vero}(b).  We obtain therefore that $S_\sigma(V_\ell)\cong  A(2,\frac{\ell}{2})$ as $U_q(sl_2)$-module algebras. 
Now,  Theorem \ref{th:main}  (b) is a direct consequence of Lemma  \ref{le:Veronese splitting} and 
 Lemma \ref{le:V generates sub Vero}(a). Theorem \ref{th:main} is proved. 
\end{proof}

\begin{remark}
Numerical experiments and the philosophy of \cite{ZW} lead us to conjecture that if $\gg=sl_n$, a similar description of the braided symmetric algebras of $V_{2k\omega_1}$ in terms of Veronese algebras of $V_{2\omega_1}$ can be obtained. Here $\omega_1$ denotes the first fundamental weight of $\gg=sl_n$. This is particularly interesting, as $V_{2\omega_1}=S^2(V_{\omega_1})$ is  a flat $U_q(sl_n)$-module (see \cite[Theorem 1.1]{ZW})\end{remark}
 \subsection{The Radical}
In this section we will consider the set of nilpotent elements in the algebras $S_\sigma (V_\ell)$. First note the following facts.
\begin{proposition}
\label{pr:} 
 \noindent(a) Let $\ell$ be even. Then $0$ is the only nilpotent element in $S_\sigma(V_\ell)$.
 
 \noindent(b) Let $\ell$ be odd. The set of nilpotent elements $\mathcal{N}$ is a  graded $U_q(sl_2)$-module ideal of $S_\sigma (V_\ell)$. More precisely, 
 
 $$\mathcal{N}\cong \bigoplus_{n=2}^{\infty} \mathcal{N}_n \ ,{\rm where}  \quad \mathcal{N}_n \cong   \bigoplus_{i=1}^{\frac{\ell-1}{2}}  V_{n\ell-4i}\ .$$ 

 \end{proposition}
\begin{proof}
Part (a) follows from the description of $S_\sigma (V_\ell)$ as a subalgebra of the quantum Veronese algebra $\mathcal{V}(2,\frac{\ell}{2})$ which   has no zero divisors. In order to prove part(b), we first show that  the submodules $\mathcal{N}$ is contained in the radical $Rad(S_\sigma (V_\ell))$.  Notice that $\mathcal{N}$ is an ideal, hence it suffices to prove that all elements of $\mathcal{N}$  are nilpotent. Consider  ${\bf v}\in  \mathcal{N}_n$. If ${\bf v}^\ell \ne 0$, then ${\bf v}^\ell\in S_\sigma (V_\ell)_{n\ell}$ lies in a submodule   which contains no highest weight vector of weight greater than $(\ell n\ell-4\ell)$. But Theorem \ref{th:main} implies that there exists no such highest weight vector in $S_\sigma (V_\ell)_{n\ell}$, hence ${\bf v}^\ell = 0$.  

Observe that the quotient $S_\sigma(V_\ell)/ \mathcal{N}$ is isomorphic to the Veronese algebra $V_q(1,\ell)$ defined in Definition \ref{def:Veronese}(b).  Since $\mathcal{N}\subset Rad(S_\sigma (V_\ell))$ we obtain that $Rad(S_\sigma (V_\ell))/\mathcal{N}\subset Rad (V_q(1,\ell))$. But $Rad (V_q(1,\ell))=\{0\}$ by Lemma \ref{le:radical veronese}. Hence $Rad(S_\sigma (V_\ell))=\mathcal{N}$.
   Part(b) is proved. The proposition is proved.\end{proof}

\section{Applications}
\label{se:Applications}

\subsection{The Poisson Closure}
\label{se:Poisson closure}
In this section, we will connect the results on braided symmetric and exterior powers to the semiclassical limits, and the notion of the Poisson closure, as developed in \cite{BZ}. Recall from Theorem \ref{th: Poisson-closure} (b) that the semiclassical limit  $\overline{S_\sigma(V)}$ of $S_\sigma(V)$ admits a Poisson bracket defined as $\{u,v\}=r^-(u\wedge v)$. Recall the  {\it  Jacobian map} from \cite{BZ} and \cite{ZW}:
$$ J:  S(\overline V)\cdot S(\overline V)\cdot S(\overline V)  \to S(\overline V) ,$$
$$ u\otimes v\otimes  w\mapsto \{u,\{v,w\}\}+ \{w,\{u,v\}\}+ \{v,\{w,u\}\}\ 
 .$$
   Indeed, we showed that the image of $J$ is a graded  $U(\gg)$-module ideal and that 
 $S(\overline V)/J  $ together with the bracket induced by $r^-$ 
 is a Poisson algebra, which we refer to as the {\it Poisson closure} of $(S(\overline V),\{\cdot,\cdot\})$ (Theorem \ref{th: Poisson-closure} (b) and \cite[Corollary 2.20]{BZ}).
 
 Using the construction of the classical limit developed in Section \ref{se:classical limit} we now derive the following corollary from Theorem \ref{th:main}.
 
 \begin{corollary}
 Let $V$ be a simple $U_q(sl_2)$-module. The classical limit of the braided symmetric algebra $\overline{S_\sigma (V)}$ is isomorphic, as an module algebra, to the Poisson closure of $S (\overline V)$.
 \end{corollary}
 
Employing the notion of bracketed superalgebras (see \cite[Section 2.2]{BZ}), an analogous assertion for the braided exterior algebras of  simple $U_q(sl_2)$-module was proved in \cite{BZ}. Thus, our result gives further evidence for the following conjecture.
 
\begin{conjecture}
Let $V$ be an  object in $\OO_f$. The classical limit of the braided symmetric, resp.~exterior algebra $\overline{S_\sigma (V)}$, resp.~$\overline{\Lambda_\sigma (V)}$  is isomorphic, as an algebra, to the Poisson closure, of $S(\overline V)$, resp.~$\Lambda (\overline V)$.
\end{conjecture}

\subsection{Non-Commutative Geometry}
\label{se:non-comm geom}
In this section we will associate the braided symmetric algebras with non-commutative geometry in the sense of Artin, Tate and Van den Bergh \cite{ATV} and Artin-Zhang \cite{AZ 94}. First, we will recall some of their definitions and results. Let $V$ be a $n$-dimensional vectorspace. Every element of $f\in(V^*)^{\otimes k}$ defines a multilinear form  on $V^{\otimes k}$. It induces a form on $(\PP V)^{\otimes k}$, where $\PP V$ is the projective space of lines in $V$.  Since the form is multi-homogeneous,
it defines a zero locus in $(\PP V)^{\otimes k}$. Let $\II$ be a graded ideal in $T(V^*)$. Then the $d$-th graded component $\II_d$ of $\II$ defines a scheme $\Gamma_d$, where $\Gamma_d$ is the scheme  of zeros of  $\II_d\subset (\PP V)^{\otimes d}$. We have the following fact.

\begin{proposition}\cite[Proposition 3.5 ii]{ATV}
Let $\PP V_i$ denote the $i$-th factor of $(\PP V)^{\otimes d}$ and for $1\le i\le j\le d$ denote by $pr_{ij}$ the projection of  $(\PP V)^{\otimes d}$ onto $\PP V_i\times\ldots\times \PP V_j$. Then  $pr_{ij}(\Gamma_d)$ is a closed subscheme of $\Gamma_{j-i+1}$. 
\end{proposition}

Notice that by  applying projections $pr_{1(d-1)}$ to $(\PP V)^{\otimes d}$, therefore mapping $\Gamma_d$ to $\Gamma_{d-1}$, we can now consider the inverse limit $\Gamma$ of the sets $\Gamma_d$.
We need it in order to classify the  point modules which we introduce in the following.
 Denote by $A$  the algebra  $A=T(V^*)/\II$.  
\begin{definition}\cite[Definition 3.8]{ATV}
 A graded right $A$-module $M$ is called a point module if it satisfies the following conditions:
\begin{itemize}
\item $M$ is generated in degree zero,
\item $M_0=k$, and
\item ${\rm dim} M_i=1$ for all $i\ge 0$.
\end{itemize}
\end{definition}

Artin, Tate and Van Den Bergh classified the point modules in the following terms.

\begin{proposition}\cite[Corollary 3.13]{ATV}
\label{pr:class of point modules}
 The point modules of $A$ are in one-to-one correspondence with the points of $\Gamma$. 
\end{proposition}

Using our explicit description of the braided symmetric algebras of simple $U_q(sl_2)$-modules we obtain the following result.

\begin{theorem}
\label{th: non-comm curve}
\noindent(a) Let $\ell$ be odd. Then the point modules of $S_\sigma (V_\ell)$ are parametrized by the points of the curve, defined by the Veronese algebra $V_q(1,\ell)$ in $\PP V_\ell$.

\noindent(b) Let $\ell$ be even. Then the point modules of $S_\sigma (V_\ell)$ are parametrized by the points of the surface, defined by the Veronese algebra $V_q(2,\frac{\ell}{2})$ in $\PP V_\ell$.  
\end{theorem}

\begin{proof}
Part (b)  follows directly from Proposition \ref{pr:class of point modules} and the proof of  Theorem  \ref{th:main} (b)  which also provides the embedding $\PP V_\ell\subset \PP^q$ where $q=\binom{2+\frac{\ell}{2}}{\frac{\ell}{2}} $,   as the generating set of $S_\sigma (V_\ell)$, as in the proof of Theorem  \ref{th:main} (b). 

In order to prove (a) we have to observe that the quotient of $S_\sigma (V_\ell)$ by the ideal $\mathcal{N}$ is isomorphic to the Veronese algebra $V_q(1,\ell)$. The assertion follows. 
The theorem is proved.
\end{proof}

\begin{remark}
The assertion  of Theorem \ref{th: non-comm curve}  was proved in the special case $\ell=3$ by Vancliff in \cite{Van}.
\end{remark}

It will be interesting to gain a more complete view of the non-commutative geometry related to braided symmetric algebras. As mentioned in the introduction, a corresponding version of $Proj$ was defined by Artin and Zhang in \cite{AZ 94}.  Let $A$ be a non-commutative algebra, $_AA$ the left-regular module, $gr(A)$ the category of finitely  graded $A$-modules (here grading means $\ZZ$-grading) and $s$ the auto-equivalence induced by the degree-shift functor. Then, $qgr(A)$ is the quotient of the category of $gr(A)$-modules by the finite-dimensional modules (the torsion modules) and $Proj(A)$ is given by $Proj(A)=(qgr(A),_A A,s)$.  We believe that this approach will be useful in describing the geometry attached to braided symmetric algebras, and provide important and interesting avenues for future research.
\begin{appendix}
\section{Veronese Algebras--Classical and Quantum}
\label{ap: Veronese}
In this section we shall review the definitions and properties of the quantum and classical Veronese algebras and varieties. For more details and proofs of the classical facts see e.g.~\cite{Har}. Recall the definition of the {\it skew polynomials} $\CC_q[x_0,\ldots ,x_n]$ which are defined as the free algebra generated by $x_0$,\ldots $x_n$ and subject to the relations
$$x_j x_i=q^{-1} x_i x_j$$
for all $0<i<j<n$. Note that it is a flat deformation of the polynomial algebra $\CC[x_0,\ldots ,x_n]$. If $x_I=x_{i_1}\ldots x_{i_d}$ and $x_J=x_{j_1}\ldots x_{j_\ell}$ with $i_1\le \ldots i_d$ and $j_1\le \ldots j_d$, then denote by $\Lambda(I,J)$ the number
$$  \Lambda(I,J)=\sum_{m=1}^\ell {\rm max} \{k: i_k<j_m\}\ .$$
It is easy to verify that  $x_I x_J=q^{\Lambda(I,J)} x_K$, where $K=x_{k_1}\ldots x_{k_{d+\ell}}$ with $k_1\le k_2\le \ldots \le k_{d+\ell}$.
\begin{definition}
\label{def:Veronese}
\noindent(a)The Veronese algebra $V(n,d)$ is the subalgebra of $\CC[x_0,\ldots,x_n]$ generated by the homogeneous elements of degree $d$.

\noindent(b)The quantum Veronese algebra $V_q(n,d)$ is the subalgebra of $\CC_q[x_0,\ldots,x_n]$ generated by the homogeneous elements of degree $d$.
\end{definition} 
The Veronese algebra induces a map from $\PP^n\to \PP^{\binom{n+d}{d}-1}$. Its image is called the  {\it Veronese variety} $\mathcal V (n,d)$. 
The relations defining the Veronese variety are given as follows. Notice that we can describe the  coordinate functions $x_I$ on $\CC^{\binom{n+d}{d}}$ by the $d$-element partitions  $I=(0\le i_1\le i_2\le \ldots \le i_d\le n)$, or rather $x_I=x_{i_1}\ldots x_{i_d}$. We have the following fact.
 
 \begin{lemma} \noindent(a) \cite[example 2.4]{Har}
The Veronese variety  $\mathcal V (n,d)$ is the zero locus of all relations
$x_Ix_J=x_Kx_L\in \CC[x_0,\ldots,x_{\binom{n+d}{d}-1}]$ such that $x_Ix_J=x_Kx_L\in \CC[x_0,\ldots,x_n]$. In particular, all  relations are quadratic.

\noindent(b) The quantum Veronese algebra is   the quotient of  $\CC_q[x_0,\ldots,x_{\binom{n+d}{d}-1}]$ by the ideal generated by the relations
$$x_I x_J=q^{\Lambda(I,J;K,L)}x_Kx_L \in \CC_q[x_0,\ldots,x_n]\ ,$$
where $\Lambda(I,J;K,L)=\Lambda(I,J)-\Lambda(K,L)$.
 \end{lemma}

 Recall that the radical $Rad(A)$ of an algebra $A$ is defined as the intersection of all maximal left modules. 
 
 \begin{lemma}
 \label{le:radical veronese}
 $Rad(\mathcal V (n,d))=\{0\}$ and $Rad(\mathcal V_q (n,d))=\{0\}$.
 \end{lemma}
 
 \begin{proof}
 It is well-known, see e.g. \cite[Ch. IX.2]{Hung} that if $A$ is an algebra and $a\in Rad(A)$, then $1+a$ is left invertible. But if $x\in\mathcal V (n,d) $, resp.~$x\in\mathcal V_q(n,d) $ and $x+1$ is left invertible, then $x$ is a unit or $x=0$. The lemma is proved.
 \end{proof}
 
  We will now consider the Veronese subalgebras in the case of $n=1$ and $n=2$.  Recall that $ \CC[x_0,x_1]$ and  $ \CC[x_0,x_1, x_2]$, resp.~$ \CC_q[x_0,x_1]$and  $ \CC_q[x_0,x_1, x_2]$ can be interpreted as symmetric algebras of the simple $U(sl_2)$-modules $S(\overline V_1)$ and $S(\overline V_2)$, resp.~the simple   $U_q(sl_2)$-modules $S_\sigma (V_1)$ and $S_\sigma(V_2)$. We can therefore think about  the classical and quantum  Veronese algebras as  graded $U(sl_2)$, resp.~$U_q(sl_2)$-module algebras (see \cite{BZ}). Using the well known decomposition of the symmetric powers of $V_1$ and $V_2$, we obtain the following fact.
 
 \begin{lemma}
 \label{le:Veronese splitting}
 The graded components of the Veronese algebras $V(1,d)$, $V_q(1,d)$,  $V(2,d)$, $V_q(2,d)$ split as follows:
$$V(1,d)_k\cong \overline V_{kd}\ ,\quad V_q(1,d)_k\cong  V_{kd}\ ,$$
$$ V(2,d)_k  \cong \bigoplus_{i=0}^{\frac{kd}{2}} \overline  V_{2kd-4i}\  , \quad V_q(2,d)_k  \cong \bigoplus_{i=0}^{\frac{kd}{2}} V_{2kd-4i}\ .$$
 \end{lemma}
 
 Additionally, we have the following fact, relating Veronese algebras with classical and braided symmetric algebras.
 
 \begin{lemma}
 \label{le:V gen subal of Ve}
 \noindent(a) Let $\overline V$ be a $U(sl_2)$-submodule of $V(n,d)_1$. Then there exists a surjective $U(sl_2)$-module homomorphism from the symmetric algebra $S(\overline V)$ onto the subalgebra of $V_{n,d}$ generated by $\overline V$.

 \noindent(b) Let $V$ be a $U_q(sl_2)$-submodule of $V_q(n,d)_1$, $n=1,2$. Then there exists a surjective $U_q(sl_2)$-module homomorphism from the braided symmetric algebra $S_\sigma( V)$ onto the subalgebra of $V_q(n,d)$ generated by $V$.
 \end{lemma}
 
 \begin{proof}
 We first prove (b), and the proof of (a) will be analogous. Notice that if $V\subset V_q(n,d)_1$, then we can embed $V$ in $V_{n}^{\otimes kd}$. Recall the alternate definition of  $S_\sigma(V)$ as in Remark \ref {re: BSA via braiding}. Using well-known facts about braid groups, we  write $\RR_{12\ldots kd,(kd+1)\ldots 2kd}$ acting on $V_{n}^{\otimes kd}\otimes V_{n}^{\otimes kd}$ as a product of $\RR_{i,i+1}$. Because $S_\sigma(V_n)$ is flat when $n=1,2$, we see that if  ${\bf v}\in V_{n}^{\otimes kd}\otimes V_{n}^{\otimes kd}$  is in the subspace spanned by the "positive" eigenvectors for each  $\RR_{i,i+1}$ then it can be expressed  as a sum of eigenvectors with positive eigenvalue for  $\RR_{12\ldots kd,(kd+1)\ldots 2kd}$.  The  homomorphism $T(V)\to  V_q(n,d)$ therefore factors through $S_\sigma(V)$. Part (b) is proved.
 Part(a) follows by an analogous argument, as the symmetric algebras $S(\overline V_n)$ are trivially flat.
 \end{proof}
 
 Notice that $\overline V_{2d}\subset  V(2,d)_1$ and  $V_{2d}\subset  V_q(2,d)_1$ . We have the following fact about the subalgebras $A(2,d)$, resp.~$A_q(2,d)$ generated by  $\overline V_{2d}$ and  $V_{2d}$.
 
\begin{lemma}
\label{le:V generates sub Vero}
\noindent(a) The subalgebra $A(2,d)\subset  V(2,d)$ is a graded $U(sl_2)$-module algebra and $A_{q}(2,d)\subset  V_q(2,d)$ is a graded $U_q(sl_2)$-module algebra. If $k\ge 2$ then
$$A(2,d)_k= V(2,d)_k\ ,\quad  \text{and} \quad  A_q(2,d)_k= V_q(2,d)_k\ .$$

\noindent(b) The ideal of relations of the natural projection map $T(\overline V_\ell)\to A(2,d)$, resp.~$T( V_\ell)\to A_q(2,d)$ is generated by quadratic and cubic elements.

\end{lemma} 

\begin{proof}
Prove (a) first.  We prove the assertion by induction on $k$.  First, consider the submodule $\overline V_{2d}\subset V_q(2,d)_1$ of resp.~$V_{2d}\subset V_q(2,d)_1$. It is generated by $\overline x_0^d$, resp.~$x_0^d$. We have the following fact which can be proved by an easy induction argument using the definitions of Section \ref{se:q-group}. 
 
 \begin{claim}
 \label{cl:F action on S}
 \noindent (a) For each $m\ge0$, we have
 $$F^m(\overline x_0^d)=\sum_{i,j\ge 0} \overline c_{ij} \overline x_0^i\overline x_1^{d-i-j}\overline x_2^j \ ,$$
  $$F^m(x_0^d)=\sum_{i,j} c_{ij} x_0^ix_1^{d-i-j}x_2^j \ ,$$
 where $\overline c_{ij}\in  \ZZ_{>0}$ resp.~the classical limit  of $c_{ij}$ is a positive integer, if and only if  $2i-2j=2d-2m$ (i.e.~the monomial lies in the correct weight-space).   Otherwise $\overline c_{ij}=c_{ij}=0$.
 
 \noindent(b)There exists a unique solution $(i,j)\in \ZZ^2_{\ge 0}$  of the system of equations $i+j=d$ and $i-j=d-m$ if and only if $m$ is even. That means that if $m$ is odd, then $F^m(x_0^d)$ is divisible by $x_1$.
 \end{claim}
    
  Secondly, we also need the following fact which is easy to prove.
  
  \begin{claim}
  \label{cl:hw in S(V2)}
  The highest weight vectors ${\bf \overline v_{m}}$ in $S(\overline V_2)_d$,  resp.~${\bf v_{m}}$ in $S_\sigma( V_2)_d$ of weight $2d-2m$ are of the form 
  $${\bf \overline v_{m}}=\sum_{i=0}^m (-1)^i \binom{m}{i} \overline x_0^{d-m-i}\ \overline x_1^{2i}\ \overline x_2^{m+i} \ ,$$
  $${\bf  v_{m}}=\sum_{i=0}^m (-1)^i \binom{m}{i}_q x_0^{d-m-i} x_1^{2i} x_2^{m+i} \ .$$
  \end{claim}

 Now consider the products $ \overline  V_{2d}\cdot  \overline V_{2d} \subset V(2,d)_2$ and $V_{2d}\cdot V_{2d} \subset V_q(2,d)_2$. It is easy to see that  
 \begin{equation}
 \label{eq:evenhw in tensors1}
E ({\bf \overline v_{i}})=E\left(\sum_{j=0}^i  (-1)^j \binom{i}{j}  \overline v_j  \overline v_{2i-j}\right)=0\ ,\end{equation}
\begin{equation}
 \label{eq:evenhw in tensors2}
 E ({\bf  v_{i}})=E\left(\sum_{j=0}^i  (-1)^j \binom{i}{j}_q  v_j v_{2i-j}\right)=0\ 
 \end{equation}
 we have to verify, however, that  ${\bf \overline v_{i}}\ne 0\in V(2,d)$, resp.~${\bf v_{i}}\ne 0\in V_q(2,d)$. Claim \ref{cl:F action on S}  (a) and (b) and Equation \eqref{eq:evenhw in tensors1} allow us to determine that  
 $${\bf \overline v_{i}}=\sum_{j=0}^i  (-1)^j \binom{i}{j}  \overline v_j\overline v_{2i-j}  =  c\overline x_0^{2d-i}\overline x_2^i +\overline x_1(\ldots)  \ ,$$
 $$  {\bf  v_{i}}=\sum_{j=0}^i  (-1)^j \binom{i}{j}_q  v_j v_{2i-j}  =  c x_0^{2d-i} x_2^i +x_1(\ldots)$$
 with non-zero $c$, because of the following argument. Terms of the form  $\overline x_0^{d-k}\overline x_2^k$, resp.~$x_0^{d-k} x_2^k$ appear only in the weight vectors $\overline v_{2k}$, resp.~$v_{2k}$ in $V_{2d}$. But these terms appear only in summands with positive sign in \eqref{eq:evenhw in tensors1} resp.~\eqref{eq:evenhw in tensors2}.  This implies that $c\ne 0$. Hence the ${\bf \overline v_{i}}$, resp.~${\bf  v_{i}}$ are  highest weight vectors of the desired weight and $\overline V_{2d}\cdot \overline V_{2d}$ , therefore, generates $V(2,d)_2$, while $V_{2d}\cdot V_{2d}$ generates $V_q(2,d)_2$.
 
 In the case of $V(2,d)_{k}$ and $V_q(2,d)_{k}$, $k\ge 3$, assume by induction that the assertion has been proven for $k-1$. We observe that we can express any monomial  $\overline x_{0}^i \ \overline x_{1}^j\  \overline x_2^{kd-i-j}$, resp.~$x_{0}^i x_{1}^j x_2^{kd-i-j}$ as 
 $$\overline  x_{0}^i\  \overline  x_{1}^j \ \overline  x_2^{kd-i-j}=\sum_{m=0}^{d} a_m \cdot  \overline  v_m \ ,$$
 $$x_{0}^i x_{1}^j x_2^{kd-i-j}=\sum_{m=0}^{d} a_m \cdot v_m \ ,$$
 where $a_m\in V(2,d)_{k-1}$ and $v_m\in V_{2d}$ as defined above. If $ i\ge 2d$ or $kd-i-j\ge 2d$, then it is obvious as  $a_i=0$ for all $i>0$, resp.~ $a_i=0$ for all $i<d$. In the remaining cases, we will set up a system of equations as follows. We write
  $$\overline x_{0}^i\ \overline  x_{1}^j\overline\  x_2^{kd-i-j}=\sum_{m=0}^{d-i} a_m \cdot \overline v_m \ ,$$
 $$x_{0}^i x_{1}^j x_2^{kd-i-j}=\sum_{m=0}^{d-i} a_m \cdot v_m \ ,$$
 and it is easy to see that we can solve the resulting system of equations for the $a_m$.
  Lemma \ref{le:V generates sub Vero}(a) follows. 
  
 Now consider (b). Recall that $A(2,d)=T(\overline{V_{2d}})/I$ where $I$ is a homogeneous ideal.  Since $A(2,d)_k=V(2,d)_k$, for $k\ge 2$ and  $A(2,d)_k\cdot A(2,d)_\ell=V(2,d)_m$ for all $k+\ell=m$ with $k,\ell\ge 2$ and since there are only quadratic relations in $V(2,d)$, the ideal $I$ is generated by homogeneous elements of  degree  at most three.  One argues analogously in the case of $A_q(2,d)$ and $V_q(2,d)$.   Part(b) and, therefore,  Lemma \ref{le:V generates sub Vero} are proved.   \end{proof}
 \end{appendix}

\end{document}